\documentclass[12pt]{amsart}
\usepackage{amsmath,amssymb,latexsym,cancel,rotating}
\usepackage{graphicx,amssymb,mathrsfs,amsmath,color,fancyhdr,amsthm}
\usepackage[all]{xy}

\newtheorem{theorem}{Theorem}[section]
\newtheorem{lemma}[theorem]{Lemma}
\newtheorem{corollary}[theorem]{Corollary}

\newtheorem{proposition}[theorem]{Proposition}

\newtheorem{thm}{Theorem}

 \def\cD{{\mathcal D}}    \def\cL{{\mathcal L}}

\def\bbN{{\mathbb N}}  \def\bbZ{{\mathbb Z}}  \def\bbQ{{\mathbb Q}}

\def\bbK{{\mathbb K}}

  \def\leq{\leqslant}  \def\geq{\geqslant}

\def\dim{\mbox{\rm dim}\,}

\def\bfV{{\mathbf V}}

\begin{document}

\title[Geometric realizations of Lusztig's symmetries]
{Geometric realizations of Lusztig's symmetries on the whole quantum groups}\thanks{This work was supported by the National Natural Science Foundation of China [No. 11526037]}

\author[Zhao]{Minghui Zhao}
\address{College of Science, Beijing Forestry University, Beijing 100083, P. R. China}
\email{zhaomh@bjfu.edu.cn (M.Zhao)}


\date{\today}


\bibliographystyle{abbrv}

\maketitle

\begin{abstract}
In this paper, we shall study the structure of the Grothendieck group of the category consisting of Lusztig's perverse sheaves and give a decomposition theorem of it. By using this decomposition theorem and the geometric realizations of Lusztig's symmetries on the positive part of a quantum group, we shall give geometric realizations of Lusztig's symmetries on the whole quantum group.
\end{abstract}

\setcounter{tocdepth}{1}\tableofcontents

\section{Introduction}

\subsection{}

Let $\mathbf{U}$ be the quantum group associated to a Cartan datum, which is introduced by Drinfeld (\cite{drinfel1990hopf}) and Jimbo (\cite{jimbo1985aq}) respectively in the study of quantum Yang-Baxter equations. As a quantization of the universal enveloping algebra of a Kac-Moody Lie algebra, the quantum group $\mathbf{U}$ is a Hopf algebra and has the following
triangular decomposition
$$\mathbf{U}\cong {\mathbf{U}^-}\otimes{\mathbf{U}^{0}}\otimes{\mathbf{U}^{+}}.$$
In \cite{Lusztig_Introduction_to_quantum_groups}, Lusztig also introduced an algebra $\mathbf{f}$ (called the Lusztig's algebra) associated to a Cartan datum, satisfying that there are two
monomorphisms of algebras ${^+}:\mathbf{f}\rightarrow\mathbf{U}$ and ${^-}:\mathbf{f}\rightarrow\mathbf{U}$ with images $\mathbf{U}^+$ and $\mathbf{U}^-$ respectively.

Lusztig introduced some operators $T_i$ on the quantum group $\mathbf{U}$ satisfying the braid group relations, which are called Lusztig's symmetries
(\cite{Lusztig_Quantum_deformations_of_certain_simple_modules_over_enveloping_algebras,Lusztig_Quantum_groups_at_roots_of_1}). By the definition of Lusztig's symmetries, the image of $\mathbf{U}^{+}$ under $T_i$ is not contained in $\mathbf{U}^{+}$ for any $i$. So Lusztig introduced two subalgebras ${_i\mathbf{f}}=\{x\in\mathbf{f}\,\,|\,\,T_i(x^+)\in\mathbf{U}^+\}$ and ${^i\mathbf{f}}=\{x\in\mathbf{f}\,\,|\,\,T^{-1}_i(x^+)\in\mathbf{U}^+\}$ of $\mathbf{f}$ (\cite{Lusztig_Introduction_to_quantum_groups}).
Note that there exists a unique $T_i:{_i\mathbf{f}}\rightarrow{^i\mathbf{f}}$ such that $T_i(x^+)=T_i(x)^+$.

For any finite quiver $Q=(I,H)$, Ringel introduced the Ringel-Hall algebra as an algebraic model of the positive part of the corresponding quantum group (\cite{Ringel_Hall_algebras_and_quantum_groups}). Green (\cite{green1995hall}) introduced the comultiplication on the Ringel-Hall algebra and Xiao (\cite{xiao1997drinfeld}) introduced the antipode. Under these operators, the Ringel-Hall algebra has a Hopf algebra structure. Xiao also considered the Drinfeld double of the Ringel-Hall algebra (called the double Ringel-Hall algebra),
the composition subalgebra of which gives a realization of the whole quantum group (\cite{xiao1997drinfeld}).

By this algebraic realization of a quantum group, Ringel applied the BGP reflection functors (\cite{bernstein1973coxeter}) to give realizations of Lusztig's symmetries $T_i:{_i\mathbf{f}}\rightarrow{^i\mathbf{f}}$ (\cite{Ringel_PBW-bases_of_quantum_groups}). Then Xiao-Yang (\cite{Xiao_Yang_BGP-reflection_functors_and_Lusztig's_symmetries}) and Sevenhant-Van den Bergh (\cite{Sevenhant_Van_den_Bergh_On_the_double_of_the_Hall_algebra_of_a_quiver}) realized Lusztig's symmetries $T_i:\mathbf{U}\rightarrow\mathbf{U}$ also by using the BGP reflection functors. Indeed, this method is also available to give precise constructions of Lusztig's symmetries on a double Ringel-Hall algebra (\cite{deng2002double,Deng_Xiao}).

\subsection{}
In \cite{Lusztig_Canonical_bases_arising_from_quantized_enveloping_algebra,Lusztig_Quivers_perverse_sheaves_and_the_quantized_enveloping_algebras}, Lusztig gave a geometric realization of the Lusztig's algebra $\mathbf{f}$. Let $Q=(I,H)$ be the quiver corresponding to $\mathbf{f}$. Inspired by the algebraic realization of $\mathbf{f}$ given by Ringel, Lusztig considered the variety $E_{\mathbf{V}}$ consisting of representations with dimension vector $\nu$ of the quiver $Q$ and the category $\mathcal{Q}_{\mathbf{V}}$ of some semisimple complexes on $E_{\mathbf{V}}$.
Let $K(\mathcal{Q}_{\mathbf{V}})$ be the Grothendieck group of $\mathcal{Q}_{\mathbf{V}}$.
Considering all dimension vectors,
let
$$K(\mathcal{Q})=\bigoplus_{\nu}K(\mathcal{Q}_{\mathbf{V}}),$$
which is isomorphic to the Lusztig's algebra $\mathbf{f}$.

In \cite{Xiao_Xu_Zhao_Ringel_Hall_algebras_beyond_their_quantum_groups_I}, Xiao, Xu and Zhao considered a larger category of Weil complexes on the variety $E_{\mathbf{V}}$ for any dimension vector $\nu$. They showed that the direct sum of the Grothendieck groups of these categories gives a realization of the generic Ringel-Hall algebra via Lusztig's geometric method. They also considered the Drinfeld double of the direct sum and gave a realization of the generic double Ringel-Hall algebra.

By using the method of Lusztig, Kato gave geometric realizations of Lusztig's symmetries $T_i:{_i\mathbf{f}}\rightarrow{^i\mathbf{f}}$ in the case of finite type for all $i$ (\cite{Kato_PBW_bases_and_KLR_algebras}). Then his constructions were generalized by Xiao and Zhao to all symmetrizable Cartan datum (\cite{Xiao_Zhao_Geometric_realizations_of_Lusztig's_symmetries,Xiao_Zhao_Geometric_realizations_of_Lusztig's_symmetries_of_symmetrizable_quantum_groups}).

Assume that $i$ is a sink (resp. source) of the quiver $Q$. Xiao and Zhao considered a subvariety
${_i{E_{\mathbf{V},0}}}$ (resp. ${^i{E_{\mathbf{V},0}}}$) of $E_{\mathbf{V}}$ and a category ${_i\mathcal{Q}}_{\mathbf{V},0}$ (resp. ${^i\mathcal{Q}}_{\mathbf{V},0}$) of some semisimple complexes on ${_i{E_{\mathbf{V},0}}}$ (resp. ${^i{E_{\mathbf{V},0}}}$). They showed that $K({_i\mathcal{Q}}_{0})=\bigoplus_{\nu}K({_i\mathcal{Q}}_{\mathbf{V},0})$
 (resp. $K({^i\mathcal{Q}}_{0})=\bigoplus_{\nu}K({^i\mathcal{Q}}_{\mathbf{V},0})$) realizes
${_i\mathbf{f}}$ (resp. ${^i\mathbf{f}}$).

Let $i\in I$ be a sink of the quiver $Q$. Then $i$ is a source of $Q'=\sigma_iQ$, which is the quiver by reversing the directions
of all arrows in $Q$ containing $i$.
They defined a map $\tilde{\omega}_i:K({_i\mathcal{Q}}_{0})\rightarrow K({^i\mathcal{Q}}_{0})$, which gives a realization of Lusztig's symmetry $T_i:{_i\mathbf{f}}\rightarrow{^i\mathbf{f}}$.

\subsection{}
In this paper, we shall give a geometric realization of Lusztig's symmetry $T_i:\mathbf{U}\rightarrow\mathbf{U}$ for any $i$.

Let $Q$ be the quiver corresponding to $\mathbf{U}$. First, we shall construct a skew-Hopf pairing
$(\tilde{K}(\mathcal{Q})^+,\tilde{K}(\mathcal{Q})^-,\varphi)$, where $\tilde{K}(\mathcal{Q})^+$ and $\tilde{K}(\mathcal{Q})^-$
are two Hopf algebras by adding torus algebra $\mathbf{K}=\bigoplus_{\mu}\mathcal{A}\mathbf{k}_{\mu}$ to ${K}(\mathcal{Q})$. Let
$DK(\mathcal{Q})=DK(\mathcal{Q})(Q)$ be the quotient of the Drinfeld double of this skew-Hopf pairing
module the Hopf ideal generated by $\mathbf{k}_{\mu}\otimes\mathbf{1}-\mathbf{1}\otimes\mathbf{k}_{\mu}$.
Then $DK(\mathcal{Q})$ is isomorphic to the whole quantum group $\mathbf{U}$ and has the following triangular decomposition
\begin{displaymath}
DK(\mathcal{Q})\cong {{K}(\mathcal{Q})^-}\otimes{\mathbf{K}}\otimes{{K}(\mathcal{Q})^+}.
\end{displaymath}

Then, we shall study the structure of ${K}(\mathcal{Q})$. Assume that $i$ is a sink of the quiver $Q$. We have the following theorem.
\begin{thm}\label{1}
The $\mathcal{A}$-module $K(\mathcal{Q})$ has the following direct sum decomposition
$$K(\mathcal{Q})\cong\bigoplus_{r\geq0}[\mathcal{L}_{ri}]\ast ({_ij}_{0})_!(K({_i\mathcal{Q}}_{0})).$$
\end{thm}
This theorem is a geometric interpretation of the following direct sum decomposition
$$\mathbf{f}=\bigoplus_{t\geq0}\theta^{(t)}_i{_i\mathbf{f}}.$$
When $i$ is a source, we have a similar result.

Assume that $i$ is also a sink of the quiver $Q$. So $i$ is a source of $Q'=\sigma_iQ$.
By using Theorem \ref{1} and the map $\tilde{\omega}_i:K({_i\mathcal{Q}}_{0})\rightarrow K({^i\mathcal{Q}}_{0})$, we can define a map $$\tilde{T}_i:DK(\mathcal{Q})(Q)\rightarrow DK(\mathcal{Q})(Q').$$
Then we have the following main theorem in this paper.
\begin{thm}
The map $\tilde{T}_i:DK(\mathcal{Q})(Q)\rightarrow DK(\mathcal{Q})(Q')$ is an isomorphism of Hopf algebras satisfying that
the following diagram is commutative
$$
\xymatrix{
{DK(\mathcal{Q})(Q)}\ar[d]^{\cong}\ar[r]^{\tilde{T}_i}&{DK(\mathcal{Q})(Q')}\ar[d]^{\cong}\\
{\mathbf{U}_{\mathcal{A}}}\ar[r]^{T_i}&{\mathbf{U}_{\mathcal{A}}}.}
$$
\end{thm}

\section{Quantum groups and Lusztig's symmetries}

In this section, we shall recall the definitions of quantum groups and Lusztig's symmetries. We shall follow the notations in \cite{Lusztig_Introduction_to_quantum_groups}.

\subsection{}
Let $I$ be a finite index set with $|I|=n$, $A=(a_{ij})_{i,j\in I}$ be a symmetric generalized Cartan matrix, and $(A,\Pi,\Pi^{\vee},P,P^{\vee})$
be a Cartan datum associated with
$A$, where $\Pi=\{\alpha_i\,\,|\,\,i\in I\}$ is the set of simple roots, $\Pi^{\vee}=\{h_i\,\,|\,\,i\in I\}$ is the set of simple coroots, $P$ is the weight lattice and $P^{\vee}$ is the dual weight lattice. There is a symmetric bilinear form $(-,-)$ on $\mathbb{Z}I$ induced by $(i,j)=\alpha_j(h_i)=a_{ij}$.
In this paper, assume that $P^{\vee}=\mathbb{Z}\Pi^{\vee}$ and the symmetric bilinear form on $P^{\vee}$ induced by $(h_i,h_j)=(i,j)$ is also denoted by $(-,-)$.

The quantum group $\mathbf{U}$ associated with a Cartan datum $(A,\Pi,\Pi^{\vee},P,P^{\vee})$ is an associative algebra over $\mathbb{Q}(v)$ with unit element $\mathbf{1}$, generated by the elements $E_i$, $F_i(i\in I)$ and $K_{\mu}(\mu\in P^{\vee})$ subject to the following relations
\begin{enumerate}
  \item[(1)]$K_{0}=\mathbf{1}$ and $K_{\mu}K_{\mu'}=K_{\mu+\mu'}$ for all $\mu,\mu'\in P^{\vee}$,
  \item[(2)]$K_{\mu}E_{i}K_{-\mu}=v^{\alpha_i(\mu)}E_i$ for all $i\in I$, $\mu\in P^{\vee}$,
  \item[(3)]$K_{\mu}F_{i}K_{-\mu}=v^{-\alpha_i(\mu)}F_i$ for all $i\in I$, $\mu\in P^{\vee}$,
  \item[(4)]$E_iF_j-F_jE_i=\delta_{ij}\frac{K_{i}-K_{-i}}{v-v^{-1}}$ for all $i,j\in I$,
  \item[(5)]$\sum_{k=0}^{1-a_{ij}}(-1)^{k}E_i^{(k)}E_jE_i^{(1-a_{ij}-k)}=0$ for all $i\neq j\in I$,
  \item[(6)]$\sum_{k=0}^{1-a_{ij}}(-1)^{k}F_i^{(k)}F_jF_i^{(1-a_{ij}-k)}=0$ for all $i\neq j\in I$,
\end{enumerate}
where $K_i=K_{h_i}$, $[n]_{v}=\frac{v^n-v^{-n}}{v-v^{-1}}$, $E_i^{(n)}=E_i^n/[n]_{v}!$ and $F_i^{(n)}=F_i^n/[n]_{v}!$.

The comultiplication $\Delta:\mathbf{U}\rightarrow\mathbf{U}\otimes\mathbf{U}$ is an algebra homomorphism satisfying that
\begin{enumerate}
  \item[(1)]$\Delta(E_i)=E_i\otimes\mathbf{1}+K_i\otimes E_i$ for all $i\in I$,
  \item[(2)]$\Delta(F_i)=F_i\otimes K_{-i}+\mathbf{1}\otimes F_i$ for all $i\in I$,
  \item[(3)]$\Delta(K_{\mu})=K_{\mu}\otimes K_{\mu}$ for all $\mu\in P^{\vee}$.
\end{enumerate}

The antipode $S:\mathbf{U}\rightarrow\mathbf{U}^{op}$ is an algebra homomorphism satisfying that
\begin{enumerate}
  \item[(1)]$S(E_i)=-K_{-i}E_i$ for all $i\in I$,
  \item[(2)]$S(F_i)=-F_iK_{-i}$ for all $i\in I$,
  \item[(3)]$S(K_{\mu})=K_{-\mu}$ for all $\mu\in P^{\vee}$.
\end{enumerate}

The counit $\mathbf{e}:\mathbf{U}\rightarrow\mathbb{Q}(v)$ is also an algebra homomorphism satisfying that
\begin{enumerate}
  \item[(1)]$\mathbf{e}(E_i)=0$ for all $i\in I$,
  \item[(2)]$\mathbf{e}(F_i)=0$ for all $i\in I$,
  \item[(3)]$\mathbf{e}(K_{\mu})=1$ for all $\mu\in P^{\vee}$.
\end{enumerate}

\begin{theorem}[\cite{Lusztig_Introduction_to_quantum_groups}]
The algebra $(\mathbf{U},\Delta,S,\mathbf{e})$ is a Hopf algebra.
\end{theorem}

The quantum group $\mathbf{U}$ has the following triangular decomposition
\begin{displaymath}
\mathbf{U}\cong {\mathbf{U}^-}\otimes{\mathbf{U}^{0}}\otimes{\mathbf{U}^{+}},
\end{displaymath}
where $\mathbf{U}^-$, $\mathbf{U}^+$ and $\mathbf{U}^{0}$ are the subalgebras $\mathbf{U}$ generated by $F_i$, $E_i$ and $K_{\mu}$
for all $i\in I$ and $\mu\in P^{\vee}$ respectively.

\subsection{}

In \cite{Lusztig_Introduction_to_quantum_groups}, Lusztig also introduced an associative $\mathbb{Q}(v)$-algebra $\mathbf{f}$, which is generated by $\theta_i(i\in I)$ subject to the quantum Serre ralations $\sum_{k=0}^{1-a_{ij}}(-1)^{k}\theta_i^{(k)}\theta_j\theta_i^{(1-a_{ij}-k)}=0$ for all $i\not=j\in I$,
where
$\theta_i^{(n)}=\theta_i^n/[n]_{v}!$.
There are two well-defined $\mathbb{Q}(v)$-algebra monomorphisms ${^+}:\mathbf{f}\rightarrow\mathbf{U}$ and ${^-}:\mathbf{f}\rightarrow\mathbf{U}$ satisfying $E_i=\theta_i^+$ and $F_i=\theta_i^-$ for all $i\in I$ and the images of ${^+}$ and ${^-}$ are $\mathbf{U}^+$ and $\mathbf{U}^-$ respectively.

Define $|\theta_i|=i\in\mathbb{N}I$ for any $i\in I$ and $|xy|=|x|+|y|$ by induction.
There exists a unique algebra homomorphism $r:\mathbf{f}\rightarrow\mathbf{f}\otimes\mathbf{f}$ such that $r(\theta_i)=\theta_i\otimes\mathbf{1}+\mathbf{1}\otimes\theta_i$ for any $i\in I$,
where the multiplication on $\mathbf{f}\otimes\mathbf{f}$ is defined as
$(x\otimes y)(x'\otimes y')=v^{(|y|,|x'|)}xx'\otimes yy'$.

\subsection{}

Denote by $T_i:\mathbf{U}\rightarrow\mathbf{U}$ the Lusztig's symmetries for all $i\in I$
(\cite{Lusztig_Quantum_deformations_of_certain_simple_modules_over_enveloping_algebras,Lusztig_Quantum_groups_at_roots_of_1,Lusztig_Introduction_to_quantum_groups}).
The formulas of $T_i$ on the generators are
\begin{enumerate}
  \item[(1)]$T_i(E_i)=-F_i K_{i}$ and $T_i(F_i)=-K_{-i}E_i$,
  \item[(2)]$T_i(E_j)=\sum_{r+s=-a_{ij}}(-1)^rv^{-r}E_i^{(s)}E_jE_i^{(r)}$ for $i\neq j\in I$,
  \item[(3)]$T_i(F_j)=\sum_{r+s=-a_{ij}}(-1)^rv^{r}F_i^{(r)}F_jF_i^{(s)}$ for $i\neq j\in I$,
  \item[(4)]$T_i(K_{\mu})=K_{\mu-\alpha_{i}(\mu)h_i}$ for all $\mu\in P^{\vee}$.
\end{enumerate}

Let ${_i\mathbf{f}}=\{x\in\mathbf{f}\,\,|\,\,T_i(x^+)\in\mathbf{U}^+\}$ and ${^i\mathbf{f}}=\{x\in\mathbf{f}\,\,|\,\,T^{-1}_i(x^+)\in\mathbf{U}^+\}$, which are subalgebras of $\mathbf{f}$ (\cite{Lusztig_Introduction_to_quantum_groups}).
By the definitions, there exists a unique $T_i:{_i\mathbf{f}}\rightarrow{^i\mathbf{f}}$ such that
$T_i(x^+)=T_i(x)^+$.

\begin{theorem}[\cite{Lusztig_Introduction_to_quantum_groups}]\label{theorem:decom}
The algebra  $\mathbf{f}$ has the following direct sum decompositions
$$\mathbf{f}=\bigoplus_{t\geq0}\theta^{(t)}_i{_i\mathbf{f}}=\bigoplus_{t\geq0}{^i\mathbf{f}}\theta^{(t)}_i.$$
\end{theorem}

%
%

\section{Geometric realization of the Lusztig's algebra $\mathbf{f}$}

In this section, we shall recall the geometric realization of the algebra $\mathbf{f}$ given by Lusztig (\cite{Lusztig_Canonical_bases_arising_from_quantized_enveloping_algebra,Lusztig_Quivers_perverse_sheaves_and_the_quantized_enveloping_algebras,Lusztig_Introduction_to_quantum_groups}).

\subsection{}
A quiver $Q=(I,H,s,t)$ consists of a vertex set $I$, an arrow set $H$, and two maps $s,t:H\rightarrow I$ such that an arrow $\rho\in H$ starts at $s(\rho)$ and terminates at $t(\rho)$. Let $h_{ij}=\#\{i\rightarrow j\}$, $a_{ij}=h_{ij}+h_{ji}$ and $\mathbf{f}$ be the Lusztig's algebra corresponding to $A=(a_{ij})$.
Let $p$ be a prime and $q$ be a power of $p$. Denote by $\mathbb{F}_q$ the finite field with $q$ elements and $\mathbb{K}=\overline{\mathbb{F}}_q$.

For a finite dimensional $I$-graded $\mathbb{K}$-vector space $\mathbf{V}=\bigoplus_{i\in I}V_i$, define
$$E_\mathbf{V}=E_{\mathbf{V},Q}=\bigoplus_{\rho\in H}\textrm{Hom}_{\mathbb{K}}(V_{s(\rho)},V_{t(\rho)}).$$
The dimension vector of $\mathbf{V}$ is defined as $\underline{\dim}\mathbf{V}=\sum_{i\in I}(\dim_{\mathbb{K}}V_{i})i\in\mathbb{N}I$, which can also be viewed as an element in the weight lattice $P$. The algebraic group $G_{\mathbf{V}}=\prod_{i\in I}GL_{\mathbb{K}}(V_i)$ acts on $E_\mathbf{V}$ naturally.

Fix a nonzero element $\nu\in\mathbb{N}I$. Let
$$Y_{\nu}=\{\mathbf{y}=(\mathbf{i},\mathbf{a})\,\,|\,\,\sum_{l=1}^{k}a_li_l=\nu\},$$
where $\mathbf{i}=(i_1,i_2,\ldots,i_k),\,\,i_l\in I$, $\mathbf{a}=(a_1,a_2,\ldots,a_k),\,\,a_l\in\mathbb{N}$.
Fix a finite dimensional $I$-graded $\mathbb{K}$-vector space $\mathbf{V}$ such that $\underline{\dim}\mathbf{V}=\nu$.
For any element $\mathbf{y}=(\mathbf{i},\mathbf{a})$,
a flag of type $\mathbf{y}$ in $\mathbf{V}$ is a sequence
$\phi=(\mathbf{V}=\mathbf{V}^k\supset\mathbf{V}^{k-1}\supset\dots\supset\mathbf{V}^0=0)$
of $I$-graded $\mathbb{K}$-vector spaces such that $\underline{\dim}\mathbf{V}^l/\mathbf{V}^{l-1}=a_li_l$.

Let $F_{\mathbf{y}}$ be the variety of all flags of type $\mathbf{y}$ in $\mathbf{V}$. For any $x=(x_{\rho})_{\rho\in H}\in E_{\mathbf{V}}$, a flag $\phi$ is called $x$-stable if $x_{\rho}(V^l_{s(\rho)})\subset{V}^l_{t(\rho)}$ for all $l$ and all $\rho\in H$. Let
$$\tilde{F}_{\mathbf{y}}=\{(x,\phi)\in E_{\mathbf{V}}\times F_{\mathbf{y}}\,\,|\,\,\textrm{$\phi$ is $x$-stable}\}$$
and $\pi_{\mathbf{y}}:\tilde{F}_{\mathbf{y}}\rightarrow E_{\mathbf{V}}$
be the projection to $E_{\mathbf{V}}$.

Let
$\overline{\mathbb{Q}}_{l}$ be the $l$-adic field and
$\mathcal{D}_{G_{\mathbf{V}}}(E_{\mathbf{V}})$ be the bounded $G_{\mathbf{V}}$-equivariant derived category of ${\overline{\bbQ}_l}$-constructible complexes on $E_{\mathbf{V}}$. For each $\mathbf{y}\in{Y}_{\nu}$, $\mathcal{L}_{\mathbf{y}}=(\pi_{\mathbf{y}})_!(\mathbf{1}_{\tilde{F}_{\mathbf{y}}})[d_{\mathbf{y}}](\frac{d_{\mathbf{y}}}{2})\in\mathcal{D}_{G_{\mathbf{V}}}(E_{\mathbf{V}})$ is a semisimple perverse sheaf, where $d_{\mathbf{y}}=\dim\tilde{F}_{\mathbf{y}}$ and $\mathcal{L}(d)$ is the Tate twist of $\mathcal{L}$.
Let $\mathcal{P}_{\mathbf{V}}$ be the set of simple perverse sheaves $\mathcal{L}$ on $E_{\mathbf{V}}$ such that $\mathcal{L}[r]$ appears as a direct summand of $\mathcal{L}_{\mathbf{y}}$ for some $\mathbf{y}\in Y_{\nu}$ and $r\in\mathbb{Z}$. Let $\mathcal{Q}_{\mathbf{V}}$ be the full subcategory of $\mathcal{D}_{G_{\mathbf{V}}}(E_{\mathbf{V}})$ consisting of all complexes which are isomorphic to finite direct sums of complexes in the set
$\{\mathcal{L}[r]\,\,|\,\,\mathcal{L}\in\mathcal{P}_{\mathbf{V}},r\in\mathbb{Z}\}$.

Let $K(\mathcal{Q}_{\mathbf{V}})$ be the Grothendieck group of $\mathcal{Q}_{\mathbf{V}}$ and
define $v^{\pm}[\mathcal{L}]=[\mathcal{L}[\pm1](\pm\frac{1}{2})]$ for any $\mathcal{L}\in\mathcal{Q}_{\mathbf{V}}$.
Then, $K(\mathcal{Q}_{\mathbf{V}})$ is a free $\mathcal{A}$-module, where $\mathcal{A}=\mathbb{Z}[v,v^{-1}]$.
Define $$K(\mathcal{Q})=\bigoplus_{\nu\in\mathbb{N}I}K(\mathcal{Q}_{\mathbf{V}}).$$

Let $\mathbf{B}_{\nu}=\{[\mathcal{L}]\,\,|\,\,\mathcal{L}\in\mathcal{P}_{\mathbf{V}}\}$
and $\mathbf{B}=\sqcup_{\nu\in\mathbb{N}I}\mathbf{B}_{\nu}$.
Then $\mathbf{B}$ is the canonical basis of $K(\mathcal{Q})$ introduced by Lusztig in \cite{Lusztig_Canonical_bases_arising_from_quantized_enveloping_algebra,Lusztig_Quivers_perverse_sheaves_and_the_quantized_enveloping_algebras}.

\subsection{}

For $\nu,\nu',\nu''\in\mathbb{N}I$ such that $\nu=\nu'+\nu''$ and three $I$-graded $\mathbb{K}$-vector spaces $\mathbf{V}$, $\mathbf{V}'$, $\mathbf{V}''$ such that $\underline{\dim}\mathbf{V}=\nu$, $\underline{\dim}\mathbf{V}'=\nu'$, $\underline{\dim}\mathbf{V}''=\nu''$,
Lusztig introduced the induction functor
$${\textrm{Ind}}_{\mathbf{V}',\mathbf{V}''}^{\mathbf{V}}:\mathcal{Q}_{\mathbf{V}'}\times\mathcal{Q}_{\mathbf{V}''}\rightarrow\mathcal{Q}_{\mathbf{V}},$$
which induces the following $\mathcal{A}$-bilinear operator
\begin{eqnarray*}
\ast={\textrm{ind}}_{\mathbf{V}',\mathbf{V}''}^{\mathbf{V}}:K(\mathcal{Q}_{\mathbf{V}'})\times K(\mathcal{Q}_{\mathbf{V}''})&\rightarrow&K(\mathcal{Q}_{\mathbf{V}})\\
([\mathcal{L}']\,,\,[\mathcal{L}''])&\mapsto&[\mathcal{L}']\ast[\mathcal{L}'']=[{\textrm{Ind}}_{\mathbf{V}',\mathbf{V}''}^{\mathbf{V}}(\mathcal{L}',\mathcal{L}'')].
\end{eqnarray*}
Under these operators, $K(\mathcal{Q})$ becomes an associative $\mathcal{A}$-algebra.

For $\nu,\nu_1,\nu_2,\dots,\nu_s\in\mathbb{N}I$ such that $\nu=\nu_1+\nu_2+\cdots+\nu_s$ and $I$-graded $\mathbb{K}$-vector spaces $\mathbf{V}, \mathbf{V}_1, \mathbf{V}_2,\dots,\mathbf{V}_s$ such that $\underline{\dim}\mathbf{V}=\nu$, $\underline{\dim}\mathbf{V}_l=\nu_l$, we can define the $s$-fold version of the induction functor ${\textrm{ind}}_{\mathbf{V}_1, \mathbf{V}_2,\dots,\mathbf{V}_s}^{\mathbf{V}}$ by induction (\cite{Xiao_Xu_Zhao_Ringel_Hall_algebras_beyond_their_quantum_groups_I}).

\subsection{}


For $\nu,\nu',\nu''\in\mathbb{N}I$ such that $\nu=\nu'+\nu''$ and three $I$-graded $\mathbb{K}$-vector spaces $\mathbf{V}$, $\mathbf{V}'$, $\mathbf{V}''$ such that $\underline{\dim}\mathbf{V}=\nu$, $\underline{\dim}\mathbf{V}'=\nu'$, $\underline{\dim}\mathbf{V}''=\nu''$,
%
Lusztig introduced the restriction functor
$${\textrm{Res}}_{\mathbf{V}',\mathbf{V}''}^{\mathbf{V}}:\mathcal{Q}_{\mathbf{V}}\rightarrow\mathcal{Q}_{\mathbf{V}'}\times\mathcal{Q}_{\mathbf{V}''},$$
which induces the following $\mathcal{A}$-bilinear operator
\begin{eqnarray*}
{\textrm{res}}_{\mathbf{V}',\mathbf{V}''}^{\mathbf{V}}:K(\mathcal{Q}_{\mathbf{V}})&\rightarrow&K(\mathcal{Q}_{\mathbf{V}'})\times K(\mathcal{Q}_{\mathbf{V}''})\\
\textrm{$[\mathcal{L}]$}&\mapsto&[\textrm{Res}_{\mathbf{V}',\mathbf{V}''}^{\mathbf{V}}(\mathcal{L})].
\end{eqnarray*}
Under these operators,  we have an operator $\textrm{res}:K(\mathcal{Q})\rightarrow K(\mathcal{Q})\otimes K(\mathcal{Q})$.

For $\nu,\nu_1,\nu_2,\dots,\nu_s\in\mathbb{N}I$ such that $\nu=\nu_1+\nu_2+\cdots+\nu_s$ and $I$-graded $\mathbb{K}$-vector spaces $\mathbf{V}, \mathbf{V}_1, \mathbf{V}_2,\dots,\mathbf{V}_s$ such that $\underline{\dim}\mathbf{V}=\nu$, $\underline{\dim}\mathbf{V}_l=\nu_l$, we can define the $s$-fold version of the restriction functor ${\textrm{res}}_{\mathbf{V}_1, \mathbf{V}_2,\dots,\mathbf{V}_s}^{\mathbf{V}}$ by induction (\cite{Xiao_Xu_Zhao_Ringel_Hall_algebras_beyond_their_quantum_groups_I}).

\subsection{}

Fix $\nu,\nu_1,\nu_2,\nu',\nu''\in\mathbb{N}I$ such that $\nu=\nu_1+\nu_2=\nu'+\nu''$ and $I$-graded $\mathbb{K}$-vector spaces $\mathbf{V}, \mathbf{V}_1, \mathbf{V}_2,\mathbf{V}',\mathbf{V}''$ such that $\underline{\dim}\mathbf{V}=\nu, \underline{\dim}\mathbf{V}_1=\nu_1, \underline{\dim}\mathbf{V}_2=\nu_2,\underline{\dim}\mathbf{V}'=\nu',\underline{\dim}\mathbf{V}''=\nu''$.

\begin{proposition}[\cite{Lusztig_Quivers_perverse_sheaves_and_the_quantized_enveloping_algebras,Lusztig_Introduction_to_quantum_groups,Xiao_Xu_Zhao_Ringel_Hall_algebras_beyond_their_quantum_groups_I}]
For any $\mathcal{L}_1\in\mathcal{Q}_{\mathbf{V}_1}$ and $\mathcal{L}_2\in\mathcal{Q}_{\mathbf{V}_2}$, we have
$$\mathrm{Res}^{\mathbf{V}}_{\mathbf{V}', \mathbf{V}''}(\mathrm{Ind}^{\mathbf{V}}_{\mathbf{V}_1, \mathbf{V}_2}(\mathcal{L}_1,\mathcal{L}_2))=
\bigoplus_{\nu'_1,\nu''_1,\nu'_2,\nu''_2}
\widehat{\mathrm{Ind}^{\mathbf{V}'}_{\mathbf{V}'_1, \mathbf{V}'_2}\otimes\mathrm{Ind}^{\mathbf{V}''}_{\mathbf{V}''_1, \mathbf{V}''_2}}
(\mathrm{Res}^{\mathbf{V}_1}_{\mathbf{V}'_1, \mathbf{V}''_1}(\mathcal{L}_1),\mathrm{Res}^{\mathbf{V}_2}_{\mathbf{V}'_2, \mathbf{V}''_2}(\mathcal{L}_2)),$$
where $\mathbf{V}'_1,\mathbf{V}''_1,\mathbf{V}'_2,\mathbf{V}''_2$ are $I$-graded $\mathbb{K}$-vector spaces with dimension vectors $\nu'_1,\nu''_1,\nu'_2,\nu''_2$
such that $\nu'_1+\nu''_1=\nu_1$, $\nu'_2+\nu''_2=\nu_2$, $\nu'_1+\nu'_2=\nu'$, $\nu''_1+\nu''_2=\nu''$ and the functor $\widehat{\mathrm{Ind}^{\mathbf{V}'}_{\mathbf{V}'_1, \mathbf{V}'_2}\otimes\mathrm{Ind}^{\mathbf{V}''}_{\mathbf{V}''_1, \mathbf{V}''_2}}$ is just the twist of
$\mathrm{Ind}^{\mathbf{V}'}_{\mathbf{V}'_1, \mathbf{V}'_2}\otimes{\mathrm{Ind}}^{\mathbf{V}''}_{\mathbf{V}''_1, \mathbf{V}''_2}$.
\end{proposition}

As a corollary, we have

\begin{corollary}
The operator $\mathrm{res}:K(\mathcal{Q})\rightarrow K(\mathcal{Q})\otimes K(\mathcal{Q})$ is an algebra homomorphism with respect to the twisted multiplication on $K(\mathcal{Q})\otimes K(\mathcal{Q})$.
\end{corollary}


\begin{theorem}[\cite{Lusztig_Quivers_perverse_sheaves_and_the_quantized_enveloping_algebras,Lusztig_Introduction_to_quantum_groups}]\label{theorem:3.1}
There is a unique $\mathcal{A}$-algebra isomorphism
$$\lambda_{\mathcal{A}}:K(\mathcal{Q})\rightarrow\mathbf{f}_{\mathcal{A}}$$
such that $$\lambda_{\mathcal{A}}\otimes\lambda_{\mathcal{A}}(\mathrm{res}(x))=r(\lambda_{\mathcal{A}}(x))$$ and $\lambda_{\mathcal{A}}(\mathcal{L}_{\mathbf{y}})=\theta_{\mathbf{y}}$ for all $\mathbf{y}\in Y_{\nu}$, where $\theta_{\mathbf{y}}=\theta_{i_1}^{(a_1)}\theta_{i_2}^{(a_2)}\cdots\theta_{i_k}^{(a_k)}$ and $\mathbf{f}_{\mathcal{A}}$ is the integral form of $\mathbf{f}$.
\end{theorem}


\section{Geometric realization of the quantum group $\mathbf{U}$}

In this section, we shall define a skew-Hopf pairing and show that a quotient of the Drinfeld double of this skew-Hopf pairing is isomorphic to the quantum group $\mathbf{U}$.

Let $Q$ be a quiver and fix a Cartan datum $(A,\Pi,\Pi^{\vee},P,P^{\vee})$ with $A=(a_{ij})$, where $a_{ij}=\#\{i\rightarrow j\}+\#\{j\rightarrow i\}$.
Let $\mathbf{U}$ be the quantum group corresponding to this Cartan datum.

Let $\mathbf{K}=\bigoplus_{\mu\in P^{\vee}}\mathcal{A}\mathbf{k}_{\mu}$ be the torus algebra and set
$\tilde{K}(\mathcal{Q})^+$ be the free $\mathcal{A}$-module with the basis $\{\mathbf{k}_{\mu}[\mathcal{L}]^+\mid\mu\in P^{\vee}, [\mathcal{L}]\in \mathbf{B}\}.$
The Hopf algebra structure of $\tilde{K}(\mathcal{Q})^+$ is given by the following operations (\cite{xiao1997drinfeld,Xiao_Xu_Zhao_Ringel_Hall_algebras_beyond_their_quantum_groups_I}).
\begin{enumerate}
\item[(a)]The multiplication is defined as
    \begin{enumerate}
    \item[(1)]$[\mathcal{L}_1]^+[\mathcal{L}_2]^+=[\mathcal{L}_1\ast\mathcal{L}_2]^+$ for $\mathcal{L}_i\in\mathcal{Q}_{\mathbf{V}_i}$,
    \item[(2)]$\mathbf{k}_{\mu}[\mathcal{L}]^+\mathbf{k}_{-\mu}=v^{\alpha(\mu)}[\mathcal{L}]^+$ for $\mu\in P^{\vee}$ and $\mathcal{L}\in\mathcal{Q}_{\mathbf{V}}$ with $\underline{\dim}\mathbf{V}=\alpha$,
    \item[(3)]$\mathbf{k}_{\mu}\mathbf{k}_{\mu'}=\mathbf{k}_{\mu+\mu'}$ for all $\mu,\mu'\in P^{\vee}$.
    \end{enumerate}
\item[(b)]The comultiplication is defined as
    \begin{enumerate}
    \item[(1)]$\tilde{\Delta}([\mathcal{L}]^+)=\sum_{\nu'+\nu''=\nu}\mathrm{res}^{\mathbf{V}}_{\mathbf{V}',\mathbf{V}''}[\mathcal{L}]^+(\mathbf{k}_{\mu''}\otimes \mathbf{1})$ for $\mathcal{L}\in\mathcal{Q}_{\mathbf{V}}$, where $\underline{\dim}\mathbf{V}=\nu$, $\underline{\dim}\mathbf{V}'=\nu'$ and $\underline{\dim}\mathbf{V}''=\nu''$,
    \item[(2)]$\tilde{\Delta}(\mathbf{k}_{\mu})=\mathbf{k}_{\mu}\otimes\mathbf{k}_{\mu}$ for all $\mu\in P^{\vee}$.
    \end{enumerate}
\item[(c)]The antipode is defined as
    \begin{enumerate}
    \item[(1)]$\tilde{S}([\mathcal{L}]^+)=\sum_{r\geq1}(-1)^r\sum_{\nu_1+\cdots+\nu_r=\nu}\mathbf{k}_{-\nu}
    \mathrm{ind}^{\mathbf{V}}_{\mathbf{V}_1,\dots,\mathbf{V}_r}
    \mathrm{res}^{\mathbf{V}}_{\mathbf{V}_1,\dots,\mathbf{V}_r}[\mathcal{L}]^+$
    for $0\not\cong\mathcal{L}\in\mathcal{Q}_{\mathbf{V}}$, where $\underline{\dim}\mathbf{V}=\nu$, $\underline{\dim}\mathbf{V}_l=\nu_l$,
    \item[(2)]$\tilde{S}(\mathbf{k}_{\mu})=\mathbf{k}_{-\mu}$ for all $\mu\in P^{\vee}$.
    \end{enumerate}
\end{enumerate}

Set $\tilde{K}(\mathcal{Q})^-$ be the free $\mathcal{A}$-module with the basis $\{\mathbf{k}_{\mu}[\mathcal{L}]^-\mid\mu\in P^{\vee}, [\mathcal{L}]\in \mathbf{B}\}.$
The Hopf algebra structure of $\tilde{K}(\mathcal{Q})^-$ is given by the following operations.
\begin{enumerate}
\item[(a)]The multiplication is defined as
    \begin{enumerate}
    \item[(1)]$[\mathcal{L}_1]^-[\mathcal{L}_2]^-=[\mathcal{L}_1\ast\mathcal{L}_2]^-$ for $\mathcal{L}_i\in\mathcal{Q}_{\mathbf{V}_i}$,
    \item[(2)]$\mathbf{k}_{\mu}[\mathcal{L}]^-\mathbf{k}_{-\mu}=v^{-\alpha(\mu)}[\mathcal{L}]^-$ for $\mu\in P^{\vee}$ and $\mathcal{L}\in\mathcal{Q}_{\mathbf{V}}$ with $\underline{\dim}\mathbf{V}=\alpha$,
    \item[(3)]$\mathbf{k}_{\mu}\mathbf{k}_{\mu'}=\mathbf{k}_{\mu+\mu'}$ for all $\mu,\mu'\in P^{\vee}$.
    \end{enumerate}
\item[(b)]The comultiplication is defined as
    \begin{enumerate}
    \item[(1)]$\tilde{\Delta}([\mathcal{L}]^-)=\sum_{\nu'+\nu''=\nu}(\mathbf{1}\otimes \mathbf{k}_{-\mu''})(\mathrm{res}^{\mathbf{V}}_{\mathbf{V}',\mathbf{V}''})^{op}[\mathcal{L}]^-$ for $\mathcal{L}\in\mathcal{Q}_{\mathbf{V}}$, where $\underline{\dim}\mathbf{V}=\nu$, $\underline{\dim}\mathbf{V}'=\nu'$ and $\underline{\dim}\mathbf{V}''=\nu''$,
    \item[(2)]$\tilde{\Delta}(\mathbf{k}_{\mu})=\mathbf{k}_{\mu}\otimes\mathbf{k}_{\mu}$ for all $\mu\in P^{\vee}$.
    \end{enumerate}
\item[(c)]The antipode is defined as
    \begin{enumerate}
    \item[(1)]$\tilde{S}([\mathcal{L}]^-)=\sum_{r\geq1}(-1)^r\sum_{\nu_1+\cdots+\nu_r=\nu}
    \mathrm{ind}^{\mathbf{V}}_{\mathbf{V}_1,\dots,\mathbf{V}_r}
    (\mathrm{res}^{\mathbf{V}}_{\mathbf{V}_1,\dots,\mathbf{V}_r})^{op}
    [\mathcal{L}]^-\mathbf{k}_{\nu}$
    for $0\not\cong\mathcal{L}\in\mathcal{Q}_{\mathbf{V}}$, where $\underline{\dim}\mathbf{V}=\nu$, $\underline{\dim}\mathbf{V}_l=\nu_l$,
    \item[(2)]$\tilde{S}(\mathbf{k}_{\mu})=\mathbf{k}_{-\mu}$ for all $\mu\in P^{\vee}$.
    \end{enumerate}
\end{enumerate}

Fix an $I$-graded $\mathbb{K}$-vector space $\bfV$ with dimension vector $\nu\in\bbN I$. Given $\cL,\cL'\in \mathcal{Q}_{\mathbf{V}}$, let
$$
\{\cL,\cL'\}=\sum_{t\in\bbZ}\mathrm{dim} H^t_{G_{\bfV}}(\cL\otimes\cL',E_{\bfV})v^t.
$$
This definition can be extended to define a bilinear form $\varphi: \tilde{K}(\mathcal{Q})^+\times\tilde{K}(\mathcal{Q})^-\rightarrow \mathbb{Q}(v)$ by setting
$$
\varphi(\mathbf{k}_{\alpha}[\cL]^+,\mathbf{k}_{\beta}[\cL']^-)=v^{-(\alpha,\beta)-\nu(\beta)+\nu'(\alpha)}\{\cL,\cL'\}
$$
for $\cL\in\mathcal{Q}_{\mathbf{V}}$ and $\cL'\in\mathcal{Q}_{\mathbf{V}'}$ such that $\underline{\dim}\mathbf{V}=\nu$ and $\underline{\dim}\mathbf{V}'=\nu'$.

\begin{proposition}[\cite{Xiao_Xu_Zhao_Ringel_Hall_algebras_beyond_their_quantum_groups_I}]
The triple $(\tilde{K}(\mathcal{Q})^+,\tilde{K}(\mathcal{Q})^-,\varphi)$ is a skew-Hopf pairing.
\end{proposition}

Let $D(\tilde{K}(\mathcal{Q})^+,\tilde{K}(\mathcal{Q})^-)$ be the Drinfeld double of this skew-Hopf pairing and
$DK(\mathcal{Q})=DK(\mathcal{Q})(Q)$ be the quotient of $D(\tilde{K}(\mathcal{Q})^+,\tilde{K}(\mathcal{Q})^-)$ module the Hopf ideal generated by $\mathbf{k}_{\mu}\otimes\mathbf{1}-\mathbf{1}\otimes\mathbf{k}_{\mu}$ for all $\mu\in P^{\vee}$.
It is clear that $DK(\mathcal{Q})$ has the following triangular decomposition
\begin{equation}\label{tri-decom}
DK(\mathcal{Q})\cong {{K}(\mathcal{Q})^-}\otimes{\mathbf{K}}\otimes{{K}(\mathcal{Q})^+}.
\end{equation}

Theorem \ref{theorem:3.1} and the construction of $DK(\mathcal{Q})$ imply the following theorem.

\begin{theorem}
There exists an isomorphism of Hopf algebras
$$\lambda_{\mathcal{A}}:DK(\mathcal{Q})\rightarrow\mathbf{U}_{\mathcal{A}}$$
such that $\lambda_{\mathcal{A}}([\mathcal{L}_{\mathbf{y}}]^+)=\theta_{\mathbf{y}}^+$,
$\lambda_{\mathcal{A}}([\mathcal{L}_{\mathbf{y}}]^-)=\theta_{\mathbf{y}}^-$ and
$\lambda_{\mathcal{A}}(\mathbf{k}_{\mu})=K_{\mu}$ for all $\mathbf{y}\in Y_{\nu}$ and $\mu\in P^{\vee}$,
where $\mathbf{U}_{\mathcal{A}}$ is the integral form of $\mathbf{U}$.
\end{theorem}

\section{Structure of the $\mathcal{A}$-module $K(\mathcal{Q})$}

For the geometric definition of Lusztig's symmetries, we shall study the structure of the $\mathcal{A}$-module $K(\mathcal{Q})$ in this section.

\subsection{}

Let $Q=(I,H,s,t)$ be a quiver. Assume that $i\in I$ is a sink.
Let $\bfV$ be a finite dimensional $I$-graded $\bbK$-vector space such that $\underline{\dim}\bfV=\nu\in\bbN I$.
For any $r\in\mathbb{Z}_{\geq 0}$, consider a subvariety ${_iE}_{\bfV,r}$ of $E_\bfV$
$$
{_iE}_{\bfV,r}=\{x\in E_{\bfV}\,\,|\,\,Im(\bigoplus_{\rho\in H,t(\rho)=i}x_{\rho}) \textrm{ has codimension $r$ in $V_i$}\}.
$$

Denote by ${_ij}_{\bfV,r}:{_iE}_{\bfV,r}\rightarrow E_{\bfV}$ the natural embedding.
Let $\cD_{G_\bfV}^b({_iE}_{\bfV,r})$ be the $G_\bfV$-equivariant bounded derived category of ${\overline{\bbQ}_l}$-constructible complexes on ${_iE}_{\bfV,r}$.
Naturally, we have two functors
$({_ij}_{\mathbf{V},r})_!:\mathcal{D}_{G_{\mathbf{V}}}({_iE}_{\bfV,r})\rightarrow \mathcal{D}_{G_{\mathbf{V}}}(E_{\mathbf{V}})$ and
${_ij}^\ast_{\mathbf{V},r}:\mathcal{D}_{G_{\mathbf{V}}}(E_{\mathbf{V}})\rightarrow\mathcal{D}_{G_{\mathbf{V}}}({_iE}_{\bfV,r})$.

Note that the variety $E_\bfV$ is a disjoint union of ${_iE}_{\bfV,r}$ for all $r\geq0$. For any $r\geq0$, let ${_iE}_{\bfV,\geq r}=\cup_{r'\geq r}\,\,{_iE}_{\bfV,r'}$ and ${_iE}_{\bfV,\leq r}=\cup_{r'\leq r}\,\,{_iE}_{\bfV,r'}$.
Let $i_{\bfV,\geq r}:{_iE}_{\bfV,\geq r}\rightarrow E_{\bfV}$ be the natural closed embedding and ${_ij}_{\bfV,\leq r}:{_iE}_{\bfV,\leq r}\rightarrow E_{\bfV}$ be the natural open embedding.

For any $\mathbf{y}=(\mathbf{i},\mathbf{a})\in{Y}_{\nu}$, let
\begin{displaymath}
{{_i\tilde{F}}_{\mathbf{y},r}}=\{(x,\phi)\in {_iE}_{\bfV,r}\times F_{\mathbf{y}}\,\,|\,\,\textrm{$\phi$ is $x$-stable}\}
\end{displaymath}
and
${_i\pi}_{\mathbf{y},r}:{_i\tilde{F}}_{\mathbf{y},r}\rightarrow {_iE}_{\bfV,r}$
be the projection to ${_iE}_{\bfV,r}$.

For any $\mathbf{y}\in{Y}_{\nu}$,
${_i\mathcal{L}}_{\mathbf{y},r}=({_i\pi}_{\mathbf{y},r})_!(\mathbf{1}_{{_i\tilde{F}}_{\mathbf{y},r}})[d_{\mathbf{y}}](\frac{d_{\mathbf{y}}}{2})\in\mathcal{D}_{G_{\mathbf{V}}}({_iE}_{\bfV,r})$ is a semisimple perverse sheaf.
Let ${_i\mathcal{P}}_{\mathbf{V},r}$ be the set of simple perverse sheaves $\mathcal{L}$ on ${_iE}_{\bfV,r}$ such that $\mathcal{L}[s]$ appears as a direct summand of ${_i\mathcal{L}}_{\mathbf{y},r}$ for some $\mathbf{y}\in {Y}_{\nu}$ and $s\in\mathbb{Z}$. Let ${_i\mathcal{Q}}_{\mathbf{V},r}$ be the full subcategory of $\mathcal{D}_{G_{\mathbf{V}}}({_iE}_{\bfV,r})$ consisting of all complexes which are isomorphic to finite direct sums of complexes in the set $\{\mathcal{L}[s]\,\,|\,\,\mathcal{L}\in{_i\mathcal{P}}_{\mathbf{V},r},s\in\mathbb{Z}\}$.

Let $K({_i\mathcal{Q}}_{\mathbf{V},r})$ be the Grothendieck group of ${_i\mathcal{Q}}_{\mathbf{V},r}$
and
define $v^{\pm}[\mathcal{L}]=[\mathcal{L}[\pm1](\pm\frac{1}{2})]$ for any $\mathcal{L}\in{_i\mathcal{Q}}_{\mathbf{V},r}$.
Then, $K({_i\mathcal{Q}}_{\mathbf{V},r})$ is a free $\mathcal{A}$-module.
Let
$$K({_i\mathcal{Q}}_{r})=\bigoplus_{\nu\in\mathbb{N}I}K({_i\mathcal{Q}}_{\mathbf{V},r}).$$
Let ${_i\mathbf{B}}_{\mathbf{V},r}=\{[\mathcal{L}]\,\,|\,\,\mathcal{L}\in{_i\mathcal{P}}_{\mathbf{V},r}\}$ and
${_i\mathbf{B}}_{r}=\bigcup_{\nu}{_i\mathbf{B}}_{\mathbf{V},r}$, which is an $\mathcal{A}$-basis of $K({_i\mathcal{Q}}_{r})$.

\begin{lemma}\label{Q}
For any $I$-graded $\mathbb{K}$-vector space $\mathbf{V}$ and $0\leq r\in\mathbb{Z}$, we have $${_ij}^{\ast}_{\mathbf{V},r}(\mathcal{Q}_{\mathbf{V}})={_i\mathcal{Q}}_{\mathbf{V},r}.$$
\end{lemma}
\begin{proof}
For any $\mathbf{y}\in{Y}_{\nu}$, we have the following fiber product
$$\xymatrix{{_i\tilde{F}}_{\mathbf{y},r}\ar[r]^-{{{_i\tilde{j}}}_{\mathbf{V},r}}\ar[d]^-{{_i\pi}_{\mathbf{y},r}}&\tilde{F}_{\mathbf{y}}
\ar[d]^-{\pi_{\mathbf{y}}}\\{_iE}_{\bfV,r}\ar[r]^-{{_ij}_{\mathbf{V},r}}&E_{\mathbf{V}}.}$$
Hence
\begin{eqnarray*}{_ij}^{\ast}_{\mathbf{V},r}\mathcal{L}_{\mathbf{y}}&=&{_ij}^{\ast}_{\mathbf{V},r}(\pi_{\mathbf{y}})_!(\mathbf{1}_{{\tilde{F}}_{\mathbf{y}}})[d_{\mathbf{y}}](\frac{d_{\mathbf{y}}}{2})
=({_i\pi}_{\mathbf{y},r})_!{{_i\tilde{j}}}^{\ast}_{\mathbf{V},r}(\mathbf{1}_{{\tilde{F}}_{\mathbf{y}}})[d_{\mathbf{y}}](\frac{d_{\mathbf{y}}}{2})\\
&=&({_i\pi}_{\mathbf{y},r})_!(\mathbf{1}_{{_i\tilde{F}}_{\mathbf{y},r}})[d_{\mathbf{y}}](\frac{d_{\mathbf{y}}}{2})={_i\mathcal{L}}_{\mathbf{y},r}.\end{eqnarray*}
That is ${_ij}^{\ast}_{\mathbf{V},r}(\mathcal{Q}_{\mathbf{V}})={_i\mathcal{Q}}_{\mathbf{V},r}$.
\end{proof}

Hence, we get an $\mathcal{A}$-linear map
${_ij}^\ast_{r}:K(\mathcal{Q})\rightarrow K({_i\mathcal{Q}}_{r})$.

\begin{theorem}\label{tri}
For any $\mathcal{L}\in\mathcal{Q}_{\mathbf{V}}$ with $\dim V_i=s$, let $\mathcal{L}_{\geq r}=(i_{\bfV,\geq r})_\ast i^\ast_{\bfV,\geq r}\mathcal{L}$
and $\mathcal{L}_{r}={_ij}^\ast_{\mathbf{V},r}\mathcal{L}\in{_i\mathcal{Q}}_{\mathbf{V},r}$. Then
there exists a distinguished triangle
\begin{displaymath}
\xymatrix{
({_ij}_{\mathbf{V},r})_!\mathcal{L}_r\ar[r]&\mathcal{L}_{\geq r}\ar[r]&\mathcal{L}_{\geq r+1}\ar[r]&
}
\end{displaymath}
for any $0\leq r\leq s$.
\end{theorem}

\begin{proof}

By Lemma \ref{Q}, we have $\mathcal{L}_{r}={_ij}^\ast_{\mathbf{V},r}\mathcal{L}\in{_i\mathcal{Q}}_{\mathbf{V},r}$.

Then, we shall prove this theorem by induction.

By definition, $\mathcal{L}_{\geq 0}=(i_{\bfV,\geq 0})_\ast i^\ast_{\bfV,\geq 0}\mathcal{L}=\mathcal{L}$.
Since ${_iE}_{\bfV,\geq 1}$ is closed in ${E}_{\bfV}$, we have the following distinguished triangle
\begin{displaymath}
\xymatrix{
({_ij}_{\mathbf{V},0})_!{_ij}^\ast_{\mathbf{V},0}\mathcal{L}\ar[r]&\mathcal{L}_{\geq 0}\ar[r]&(i_{\bfV,\geq 1})_\ast i^\ast_{\bfV,\geq 1}\mathcal{L}\ar[r]&.
}
\end{displaymath}
That is
\begin{displaymath}
\xymatrix{
({_ij}_{\mathbf{V},0})_!\mathcal{L}_0\ar[r]&\mathcal{L}_{\geq 0}\ar[r]&\mathcal{L}_{\geq 1}\ar[r]&.
}
\end{displaymath}
Hence this theorem is true for $r=0$.

Since ${_iE}_{\bfV,\geq 2}$ is closed in ${E}_{\bfV}$, we have the following distinguished triangle
\begin{equation}\label{A}
\xymatrix{
({_ij}_{\mathbf{V},\leq 1})_!{_ij}^\ast_{\mathbf{V},\leq 1}\mathcal{L}_{\geq 1}\ar[r]&\mathcal{L}_{\geq 1}\ar[r]&(i_{\bfV,\geq 2})_\ast i^\ast_{\bfV,\geq 2}\mathcal{L}_{\geq 1}\ar[r]&.
}
\end{equation}

Consider the following distinguished triangle
\begin{displaymath}
\xymatrix{
({_ij}_{\mathbf{V},\leq 0})_!{_ij}^\ast_{\mathbf{V},\leq 0}\mathcal{L}\ar[r]&\mathcal{L}\ar[r]&\mathcal{L}_{\geq 1}\ar[r]&.
}
\end{displaymath}
Applying the functor $i^\ast_{\bfV,\geq 2}$ to this distinguished triangle,
we have
$i^\ast_{\bfV,\geq 2}\mathcal{L}_{\geq 1}=i^\ast_{\bfV,\geq 2}\mathcal{L}$. Hence
$$(i_{\bfV,\geq 2})_\ast i^\ast_{\bfV,\geq 2}\mathcal{L}_{\geq 1}=(i_{\bfV,\geq 2})_\ast i^\ast_{\bfV,\geq 2}\mathcal{L}=\mathcal{L}_{\geq 2}.$$
Applying the functor ${_ij}^\ast_{\mathbf{V},1}$ to this distinguished triangle,
we have
\begin{equation}\label{C}
{_ij}^\ast_{\mathbf{V},1}\mathcal{L}_{\geq 1}={_ij}^\ast_{\mathbf{V},1}\mathcal{L}.
\end{equation}

The complex ${_ij}^\ast_{\mathbf{V},\leq 1}\mathcal{L}_{\geq 1}$ is a complex on ${_iE}_{\mathbf{V},\leq1}$ and the support is in ${_iE}_{\mathbf{V},1}$.
Hence ${_ij}^\ast_{\mathbf{V},\leq 1}\mathcal{L}_{\geq 1}=({{_i\hat{j}}}_{\mathbf{V},1})_!{_ij}^\ast_{\mathbf{V},1}\mathcal{L}_{\geq 1}$,
where ${{_i\hat{j}}}_{\mathbf{V},1}$ is the embedding from ${_iE}_{\mathbf{V},1}$ to ${_iE}_{\mathbf{V},\leq1}$.
Then $$({_ij}_{\mathbf{V},\leq 1})_!{_ij}^\ast_{\mathbf{V},\leq 1}\mathcal{L}_{\geq 1}=({_ij}_{\mathbf{V},\leq 1})_!({{_i\hat{j}}}_{\mathbf{V},1})_!{_ij}^\ast_{\mathbf{V},1}\mathcal{L}_{\geq 1}=({{_ij}}_{\mathbf{V},1})_!{_ij}^\ast_{\mathbf{V},1}\mathcal{L}_{\geq 1}.$$
By (\ref{C}), we have $$({_ij}_{\mathbf{V},\leq 1})_!{_ij}^\ast_{\mathbf{V},\leq 1}\mathcal{L}_{\geq 1}=({{_ij}}_{\mathbf{V},1})_!{_ij}^\ast_{\mathbf{V},1}\mathcal{L}=({{_ij}}_{\mathbf{V},1})_!\mathcal{L}_1.$$

Hence the distinguished triangle (\ref{A}) can be rewrote as
\begin{displaymath}
\xymatrix{
({{_ij}}_{\mathbf{V},1})_!\mathcal{L}_1\ar[r]&\mathcal{L}_{\geq 1}\ar[r]&\mathcal{L}_{\geq 2}\ar[r]&,
}
\end{displaymath}
and this theorem is true for $r=1$.

Assume that this theorem is true for $r=k-1$.

Since ${_iE}_{\bfV,\geq k+1}$ is closed in ${E}_{\bfV}$, we have the following distinguished triangle
\begin{equation}\label{B}
\xymatrix{
({_ij}_{\mathbf{V},\leq k})_!{_ij}^\ast_{\mathbf{V},\leq k}\mathcal{L}_{\geq k}\ar[r]&\mathcal{L}_{\geq k}\ar[r]&(i_{\bfV,\geq k+1})_\ast i^\ast_{\bfV,\geq k+1}\mathcal{L}_{\geq k}\ar[r]&.
}
\end{equation}

Consider the following distinguished triangle
\begin{displaymath}
\xymatrix{
({_ij}_{\mathbf{V},\leq k-1})_!{_ij}^\ast_{\mathbf{V},\leq k-1}\mathcal{L}\ar[r]&\mathcal{L}\ar[r]&\mathcal{L}_{\geq k}\ar[r]&.
}
\end{displaymath}
Applying the functor $i^\ast_{\bfV,\geq k+1}$ to this distinguished triangle,
we have
$i^\ast_{\bfV,\geq k+1}\mathcal{L}_{\geq k}=i^\ast_{\bfV,\geq k+1}\mathcal{L}$. Hence
$$(i_{\bfV,\geq k+1})_\ast i^\ast_{\bfV,\geq k+1}\mathcal{L}_{\geq k}=(i_{\bfV,\geq k+1})_\ast i^\ast_{\bfV,\geq k+1}\mathcal{L}=\mathcal{L}_{\geq k+1}.$$
Applying the functor ${_ij}^\ast_{\mathbf{V},k}$ to this distinguished triangle,
we have
\begin{equation}\label{D}
{_ij}^\ast_{\mathbf{V},k}\mathcal{L}_{\geq k}={_ij}^\ast_{\mathbf{V},k}\mathcal{L}.
\end{equation}

The complex ${_ij}^\ast_{\mathbf{V},\leq k}\mathcal{L}_{\geq k}$ is a complex on ${_iE}_{\mathbf{V},\leq k}$ and the support is in ${_iE}_{\mathbf{V},k}$.
Hence ${_ij}^\ast_{\mathbf{V},\leq k}\mathcal{L}_{\geq k}=({{_i\hat{j}}}_{\mathbf{V},k})_!{_ij}^\ast_{\mathbf{V},k}\mathcal{L}_{\geq k}$,
where ${{_i\hat{j}}}_{\mathbf{V},k}$ is the embedding from ${_iE}_{\mathbf{V},k}$ to ${_iE}_{\mathbf{V},\leq k}$.
Then $$({_ij}_{\mathbf{V},\leq k})_!{_ij}^\ast_{\mathbf{V},\leq k}\mathcal{L}_{\geq k}=({_ij}_{\mathbf{V},\leq k})_!({{_i\hat{j}}}_{\mathbf{V},k})_!{_ij}^\ast_{\mathbf{V},k}\mathcal{L}_{\geq k}=({{_ij}}_{\mathbf{V},k})_!{_ij}^\ast_{\mathbf{V},k}\mathcal{L}_{\geq k}.$$
By (\ref{D}), we have $$({_ij}_{\mathbf{V},\leq k})_!{_ij}^\ast_{\mathbf{V},\leq k}\mathcal{L}_{\geq k}=({{_ij}}_{\mathbf{V},k})_!{_ij}^\ast_{\mathbf{V},k}\mathcal{L}=({{_ij}}_{\mathbf{V},k})_!\mathcal{L}_k.$$

Hence the distinguished triangle (\ref{B}) can be rewrote as
\begin{displaymath}
\xymatrix{
({{_ij}}_{\mathbf{V},k})_!\mathcal{L}_k\ar[r]&\mathcal{L}_{\geq k}\ar[r]&\mathcal{L}_{\geq k+1}\ar[r]&,
}
\end{displaymath}
and this theorem is true for $r=k$.

Hence, we have proved this theorem.
\end{proof}

\begin{lemma}\label{Q1}
Fix an $I$-graded $\mathbb{K}$-vector space $\mathbf{V}$ and $0\leq r\in\mathbb{Z}$. For any $\mathcal{L}\in {_i\mathcal{Q}}_{\mathbf{V},r}$, we have $[({_ij}_{\mathbf{V},r})_!(\mathcal{L})]\in K(\mathcal{Q}_{\mathbf{V}}).$
\end{lemma}

For the proof of this lemma, we need to review Lusztig's constructions of Hall algebras via functions (\cite{Lusztig_Canonical_bases_and_Hall_algebras}).
For any $I$-graded $\mathbb{K}$-vector space $\mathbf{V}$ and $1\leq n\in\mathbb{Z}$, let $E^{F^n}_\mathbf{V}$ and $G^{F^n}_{\mathbf{V}}$ be the sets consisting of the $F^n$-fixed points in $E_\mathbf{V}$ and $G_{\mathbf{V}}$ respectively, where $F$ is the Frobenius morphism.

Lusztig defined $\underline{\mathcal{F}}^n_{\mathbf{V}}$ as the set of all $G^{F^n}_{\mathbf{V}}$-invariant $\overline{\mathbb{Q}}_{l}$-functions on $E^{F^n}_\mathbf{V}$
and we can give a multiplication on $\underline{\mathcal{F}}^n=\bigoplus_{\nu\in\mathbb{N}I}\underline{\mathcal{F}}^n_{\mathbf{V}}$
to obtain the Hall algebra.
For any $i\in I$,
let $\mathbf{V}_{ti}$ be the $I$-graded $\mathbb{K}$-vector space with dimension vector $ti$ and $f_i$ be the constant function on $E^{F^n}_{\mathbf{V}_i}$ with value $1$. Denote by $\mathcal{F}^n$ the composition subalgebra of
$\underline{\mathcal{F}}^n$ generated by $f_i$
and $\mathcal{F}^n_{\mathbf{V}}=\underline{\mathcal{F}}^n_{\mathbf{V}}\cap\mathcal{F}^n$.

For any $\mathcal{L}\in\mathcal{D}_{G_{\mathbf{V}}}(E_{\mathbf{V}})$, there is a function $\chi^n_{\mathcal{L}}:E^{F^n}_\mathbf{V}\rightarrow\overline{\mathbb{Q}}_{l}$
(Section III.12 in \cite{Kiehl_Weissauer_Weil_conjectures_perverse_sheaves_and_l'adic_Fourier_transform}).
Hence, we have the trace map
\begin{eqnarray*}
\chi^n:\mathcal{D}_{G_{\mathbf{V}}}(E_{\mathbf{V}})&\rightarrow&\underline{\mathcal{F}}^{n}_{\mathbf{V}}\\
\mathcal{L}&\mapsto&\chi^n_{\mathcal{L}}.
\end{eqnarray*}
Lusztig proved that $\chi^n(\mathcal{Q}_{\mathbf{V}})=\mathcal{F}^{n}_{\mathbf{V}}$ in \cite{Lusztig_Canonical_bases_and_Hall_algebras}.

\begin{proof}[The proof of Lemma \ref{Q1}]
Similarly to the definition of $\chi^n$,
we can get a function $f_n={_r\chi}^n_{\mathcal{L}}$ on ${_iE}^{F^n}_{\bfV,r}$ for any $\mathcal{L}\in {_i\mathcal{Q}}_{\mathbf{V},r}$ and $n\in\mathbb{Z}_{\geq1}$.
The function $f_n$ can be viewed
as a function on $E^{F^n}_{\mathbf{V}}$ and ${\chi}^n_{({_ij}_{\mathbf{V},r})_!\mathcal{L}}=f_n$.

By Lemma \ref{Q}, there exists a complex $\hat{\mathcal{L}}\in \mathcal{Q}_{\mathbf{V}}$ such that ${_ij}^\ast_{\mathbf{V},r}\hat{\mathcal{L}}=\mathcal{L}$.
Note that the function $f_n$ is the restriction of ${\chi}^n_{\hat{\mathcal{L}}}$ on ${_iE}^{F^n}_{\bfV,r}$.
Since ${\chi}^n_{\hat{\mathcal{L}}}\in\mathcal{F}^n_{\mathbf{V}}$, we have $f_n\in\mathcal{F}^n_{\mathbf{V}}$.

By Theorem III.12.1 in \cite{Kiehl_Weissauer_Weil_conjectures_perverse_sheaves_and_l'adic_Fourier_transform}, there is an injective map
$$\prod_{n\in\mathbb{Z}_{\geq1}}\chi^n:K(\mathcal{D}_{G_{\mathbf{V}}}(E_{\mathbf{V}}))\rightarrow\prod_{n\in\mathbb{Z}_{\geq1}}\underline{\mathcal{F}}^{n}_{\mathbf{V}}.$$
Hence, we have $[({_ij}_{\mathbf{V},r})_!(\mathcal{L})]\in K(\mathcal{Q}_{\mathbf{V}}).$
\end{proof}

Hence, we get an $\mathcal{A}$-linear map
$({_ij}_{r})_!:K({_i\mathcal{Q}}_{r})\rightarrow K(\mathcal{Q})$.


\begin{theorem}\label{theorem:str_of_K}
For any $I$-graded $\mathbb{K}$-vector space $\mathbf{V}$ with $\dim V_i=s$, we have the following isomorphism of $\mathcal{A}$-modules
\begin{eqnarray*}
{_i\Phi}_{\mathbf{V}}:K(\mathcal{Q}_\mathbf{V})&\rightarrow&\bigoplus_{r}K({_i\mathcal{Q}}_{\mathbf{V},r})\\
\textrm{$[\mathcal{L}]$}&\mapsto&([\mathcal{L}_{0}],[\mathcal{L}_{1}],\dots,[\mathcal{L}_{s}])
\end{eqnarray*}
where $\mathcal{L}_{r}={_ij}^\ast_{\mathbf{V},r}\mathcal{L}$.
\end{theorem}
\begin{proof}
By Lemma \ref{Q}, this map is well-defined.
Consider the following map
\begin{eqnarray*}
{_i\Psi}_{\mathbf{V}}:\bigoplus_{r}K({_i\mathcal{Q}}_{\mathbf{V},r})&\rightarrow&K(\mathcal{Q}_{\mathbf{V}})\\
([\mathcal{L}_{0}],[\mathcal{L}_{1}],\dots,[\mathcal{L}_{s}])&\mapsto&[\bigoplus_{r}({_ij}_{\mathbf{V},r})_!\mathcal{L}_{r}].
\end{eqnarray*}
By Theorem \ref{tri}, we have ${_i\Psi}_{\mathbf{V}}\circ{_i\Phi}_{\mathbf{V}}([\mathcal{L}])=[\mathcal{L}]$. Hence the map ${_i\Phi}_{\mathbf{V}}$ is injective.
By Lemma \ref{Q1}, the map ${_i\Phi}_{\mathbf{V}}$ is also surjective.
\end{proof}

\begin{theorem}\label{theorem:str_of_K1}
For any $I$-graded $\mathbb{K}$-vector space $\mathbf{V}$ and integer $0\leq r\leq \dim V_i$, fix an $I$-graded vector space $\mathbf{W}$ with $\dim W_i=\dim V_i-r$ and $\dim W_j=\dim V_j$ for all $j\neq i$. Then we have
the following isomorphism of $\mathcal{A}$-modules
\begin{eqnarray*}
{_i\mathrm{ind}}_{\mathbf{W}}^{\mathbf{V}}:K({_i\mathcal{Q}}_{\mathbf{W},0})&\rightarrow&K({_i\mathcal{Q}}_{\mathbf{V},r})\\
\textrm{$[\mathcal{L}]$}&\mapsto&[{_ij}^\ast_{\mathbf{V},r}(\mathrm{Ind}_{\mathbf{V}_{ri}\mathbf{W}}^{\mathbf{V}}(\cL_{ri},({_ij}_{\mathbf{W},0})_{!}\mathcal{L}))],
\end{eqnarray*}
where $\cL_{ri}$ is the constant sheaf on $E_{\mathbf{V}_{ri}}$.
\end{theorem}

\begin{proof}
By Lemma \ref{Q1}, this map is well-defined. Consider the following map
\begin{eqnarray*}
{_i\mathrm{res}}_{\mathbf{W}}^{\mathbf{V}}:K({_i\mathcal{Q}}_{\mathbf{V},r})&\rightarrow&K({_i\mathcal{Q}}_{\mathbf{W},0})\\
\textrm{$[\mathcal{L}]$}&\mapsto&[{_ij}^\ast_{\mathbf{W},0}(\cL')],
\end{eqnarray*}
where $\mathrm{Res}_{\mathbf{V}_{ri}\mathbf{W}}^{\mathbf{V}}(({_ij}_{\mathbf{V},r})_{!}\mathcal{L})=\cL_{ri}\otimes\cL'$.
By definition, we have ${_i\mathrm{res}}_{\mathbf{W}}^{\mathbf{V}}\circ{_i\mathrm{ind}}_{\mathbf{W}}^{\mathbf{V}}=\mathrm{id}$ and ${_i\mathrm{ind}}_{\mathbf{W}}^{\mathbf{V}}\circ{_i\mathrm{res}}_{\mathbf{W}}^{\mathbf{V}}=\mathrm{id}$.
\end{proof}

By Theorem \ref{theorem:str_of_K} and \ref{theorem:str_of_K1}, we have the following decomposition of $K(\mathcal{Q})$.

\begin{theorem}\label{decom}
The $\mathcal{A}$-module $K(\mathcal{Q})$ has the following direct sum decomposition
$$K(\mathcal{Q})=\bigoplus_{r\geq0}[\mathcal{L}_{ri}]\ast ({_ij}_{0})_!(K({_i\mathcal{Q}}_{0})).$$
\end{theorem}

\begin{proposition}[\cite{Xiao_Zhao_Geometric_realizations_of_Lusztig's_symmetries,Xiao_Zhao_Geometric_realizations_of_Lusztig's_symmetries_of_symmetrizable_quantum_groups}]
There exists an isomorphism of $\mathcal{A}$-algebras ${{_i\lambda}_{0,\mathcal{A}}}:K({_i\mathcal{Q}}_0)\rightarrow{_i\mathbf{f}}_{\mathcal{A}}$ such that the following diagram is commutative
$$
\xymatrix{
K({_i\mathcal{Q}}_0)\ar[d]^-{{_i\lambda}_{0,\mathcal{A}}}\ar[r]^-{({_ij}_0)_!}&K(\mathcal{Q})\ar[r]^-{{_ij}_0^{\ast}}\ar[d]^-{\lambda_{\mathcal{A}}}    & K({_i\mathcal{Q}}_0)\ar[d]^-{{_i\lambda}_{0,\mathcal{A}}}\\
{_i\mathbf{f}}_{\mathcal{A}}\ar[r]&\mathbf{f}_{\mathcal{A}}\ar[r]^-{_i\pi_{\mathcal{A}}}& {_i\mathbf{f}}_{\mathcal{A}}.
}
$$
\end{proposition}

Hence, Theorem \ref{decom} is a geometric interpretation of Theorem \ref{theorem:decom}.

\subsection{}

When $i\in I$ is a source, we have similar constructions.
Let $\bfV$ be a finite dimensional $I$-graded $\bbK$-vector space such that $\underline{\dim}\bfV=\nu\in\bbN I$.
For any $r\in\mathbb{Z}_{\geq 0}$, consider a subvariety ${^iE}_{\bfV,r}$ of $E_\bfV$
$$
{^iE}_{\bfV,r}=\{x\in E_{\bfV}\mid Ker(\bigoplus_{\rho\in H,s(\rho)=i}x_{\rho}) \textrm{ has dimension $r$ in $V_i$}\}.
$$
Denote by ${^ij}_{\bfV,r}:{E}_{\bfV,r}\rightarrow E_{\bfV}$ be the natural embedding.
Let $\cD_{G_\bfV}^b({^iE}_{\bfV,r})$ be the $G_\bfV$-equivariant bounded derived category of ${\overline{\bbQ}_l}$-constructible complexes on ${^iE}_{\bfV,r}$.
Naturally, we have two functors
$({^ij}_{\mathbf{V},r})_!:\mathcal{D}_{G_{\mathbf{V}}}({^iE}_{\mathbf{V},r})\rightarrow \mathcal{D}_{G_{\mathbf{V}}}(E_{\mathbf{V}})$ and
${^ij}^\ast_{\mathbf{V},r}:\mathcal{D}_{G_{\mathbf{V}}}(E_{\mathbf{V}})\rightarrow\mathcal{D}_{G_{\mathbf{V}}}({^iE}_{\mathbf{V},r})$.



Similarly, we can define ${^i\mathcal{P}}_{\mathbf{V},r}$, ${^i\mathcal{Q}}_{\mathbf{V},r}$, $K({^i\mathcal{Q}}_{\mathbf{V},r})$, $K({^i\mathcal{Q}}_{r})$,
$({^ij}_r)_!$ and ${^ij}_r^{\ast}$. We also have the following results.

\begin{theorem}
For any $I$-graded $\mathbb{K}$-vector space $\mathbf{V}$ with $\dim V_i=s$, we have the following isomorphism
\begin{eqnarray*}
{^i\Phi}_{\mathbf{V}}:K(\mathcal{Q}_\mathbf{V})&\rightarrow&\bigoplus_{r}K({^i\mathcal{Q}}_{\mathbf{V},r})\\
\textrm{$[\mathcal{L}]$}&\mapsto&([\mathcal{L}_{0}],[\mathcal{L}_{1}],\dots,[\mathcal{L}_{s}]),
\end{eqnarray*}
where $\mathcal{L}_{r}={^ij}^\ast_{\mathbf{V},r}\mathcal{L}$.
The inverse of ${^i\Phi}_{\mathbf{V}}$ is
\begin{eqnarray*}
{^i\Psi}_{\mathbf{V}}:\bigoplus_{r}K({^i\mathcal{Q}}_{\mathbf{V},r})&\rightarrow&K(\mathcal{D}_{G_{\mathbf{V}}}(E_{\mathbf{V}}))\\
([\mathcal{L}_{0}],[\mathcal{L}_{1}],\dots,[\mathcal{L}_{s}])&\mapsto&[\bigoplus_{r}({^ij}_{\mathbf{V},r})_!\mathcal{L}_{r}].
\end{eqnarray*}
\end{theorem}

\begin{theorem}
For any $I$-graded $\mathbb{K}$-vector space $\mathbf{V}$ and integer $0\leq r\leq \dim V_i$, fix an $I$-graded vector space $\mathbf{W}$ with $\dim W_i=\dim V_i-r$ and $\dim W_j=\dim V_j$ for all $j\neq i$. Then we have
the following isomorphism
\begin{eqnarray*}
{^i\mathrm{ind}}_{\mathbf{W}}^{\mathbf{V}}:K({^i\mathcal{Q}}_{\mathbf{W},0})&\rightarrow&K({^i\mathcal{Q}}_{\mathbf{V},r})\\
\textrm{$[\mathcal{L}]$}&\mapsto&
[{^ij}^\ast_{\mathbf{V},r}(\mathrm{Ind}_{\mathbf{W}\mathbf{V}_{ri}}^{\mathbf{V}}(({^ij}_{\mathbf{W},0})_{!}\mathcal{L},\cL_{ri}))],
\end{eqnarray*}
whose inverse is
\begin{eqnarray*}
{^i\mathrm{res}}_{\mathbf{W}}^{\mathbf{V}}:K({^i\mathcal{Q}}_{\mathbf{V},r})&\rightarrow&K({^i\mathcal{Q}}_{\mathbf{W},0})\\
\textrm{$[\mathcal{L}]$}&\mapsto&[{^ij}^\ast_{\mathbf{W},0}(\cL')],
\end{eqnarray*}
where $\mathrm{Res}_{\mathbf{W}\mathbf{V}_{ri}}^{\mathbf{V}}(({^ij}_{\mathbf{V},r})_{!}\mathcal{L})=\cL'\otimes\cL_{ri}$.
\end{theorem}

%

\begin{theorem}\label{decom2}
The $\mathcal{A}$-module $K(\mathcal{Q})$ has the following direct sum decomposition
$$K(\mathcal{Q})=\bigoplus_{r\geq0}({^ij}_{0})_!(K({^i\mathcal{Q}}_{0}))\ast[\mathcal{L}_{ri}].$$
\end{theorem}

\begin{proposition}[\cite{Xiao_Zhao_Geometric_realizations_of_Lusztig's_symmetries,Xiao_Zhao_Geometric_realizations_of_Lusztig's_symmetries_of_symmetrizable_quantum_groups}]
There exists an isomorphism of $\mathcal{A}$-algebras ${^i\lambda}_{0,\mathcal{A}}:K({^i\mathcal{Q}}_0)\rightarrow{^i\mathbf{f}}_{\mathcal{A}}$ such that the following diagram is commutative
$$
\xymatrix{
K({^i\mathcal{Q}}_0)\ar[d]^-{{^i\lambda}_{0,\mathcal{A}}}\ar[r]^-{({^ij}_0)_!}&K(\mathcal{Q})\ar[r]^-{{^ij}_0^{\ast}}\ar[d]^-{\lambda_{\mathcal{A}}}    & K({^i\mathcal{Q}}_0)\ar[d]^-{{^i\lambda}_{0,\mathcal{A}}}\\
{^i\mathbf{f}}_{\mathcal{A}}\ar[r]&\mathbf{f}_{\mathcal{A}}\ar[r]^-{^i\pi_{\mathcal{A}}}& {^i\mathbf{f}}_{\mathcal{A}}.
}
$$
\end{proposition}

\section{Geometric realization of Lusztig's symmetry $T_i:\mathbf{U}\rightarrow\mathbf{U}$}

In this section, we shall recall the geometric realization of $T_i:{_i\mathbf{f}}\rightarrow{^i\mathbf{f}}$ in \cite{Xiao_Zhao_Geometric_realizations_of_Lusztig's_symmetries,Xiao_Zhao_Geometric_realizations_of_Lusztig's_symmetries_of_symmetrizable_quantum_groups}. By using this geometric realization and the structure of $K(\mathcal{Q})$ in last section, we shall
give a geometric realization of Lusztig's symmetry $T_i:\mathbf{U}\rightarrow\mathbf{U}$.

\subsection{}

Assume that $i$ is a sink of $Q=(I,H,s,t)$. So $i$ is a source of $Q'=\sigma_iQ=(I,H',s,t)$, where $\sigma_iQ$ is the quiver by reversing the directions
of all arrows in $Q$ containing $i$. For any $\nu,\nu'\in\mathbb{N}I$ such that $\nu'=s_i\nu$ and $I$-graded $\mathbb{K}$-vector spaces $\mathbf{V}$,  $\mathbf{V}'$ such that $\underline{\dim}\mathbf{V}=\nu$, $\underline{\dim}\mathbf{V}'=\nu'$,
there exists a functor (\cite{Xiao_Zhao_Geometric_realizations_of_Lusztig's_symmetries,Xiao_Zhao_Geometric_realizations_of_Lusztig's_symmetries_of_symmetrizable_quantum_groups})
$$\tilde{\omega}_i:{_i\mathcal{Q}}_{\mathbf{V},0}\rightarrow{^i\mathcal{Q}}_{\mathbf{V}',0},$$
which induces the following map
$$\tilde{\omega}_i:K({_i\mathcal{Q}_{0}})\rightarrow K({^i\mathcal{Q}_{0}}).$$

\begin{theorem}[\cite{Xiao_Zhao_Geometric_realizations_of_Lusztig's_symmetries,Xiao_Zhao_Geometric_realizations_of_Lusztig's_symmetries_of_symmetrizable_quantum_groups}]\label{theorem:Lusztig's_algebra}
We have the following commutative diagram
$$\xymatrix{K({_i\mathcal{Q}_0})\ar[r]^-{\tilde{\omega}_i}\ar[d]^-{_i\lambda_{0,\mathcal{A}}}&K({^i\mathcal{Q}_0})\ar[d]^-{^i\lambda_{0,\mathcal{A}}}\\
{_i\mathbf{f}}_{\mathcal{A}}\ar[r]^-{T_i}&{^i\mathbf{f}}_{\mathcal{A}}.}$$
\end{theorem}

\subsection{}

By the triangular decomposition (\ref{tri-decom}) of $DK(\mathcal{Q})(Q)$,
the following set is an $\mathcal{A}$-basis of $DK(\mathcal{Q})(Q)$
$$\{[\cL]^-\mathbf{k}_{\mu} [\cL']^+\},$$
where
$\cL\in\mathcal{P}_{\mathbf{V}}$ for some $I$-graded $\mathbb{K}$-vector space $\mathbf{V}$ with $\underline{\dim}\bfV=\nu\in\bbN I$,
$\cL'\in\mathcal{P}_{\mathbf{V}'}$ for some $I$-graded $\mathbb{K}$-vector space $\mathbf{V}'$ with $\underline{\dim}\bfV'=\nu'\in\bbN I$
and $\mu\in P^{\vee}$.

By Theorem \ref{decom}, we have
$$[\cL]=\sum_{r\geq0}[\mathcal{L}_{ri}]\ast({_ij}_{0})_![\cL_r]$$
and
$$[\cL']=\sum_{r\geq0}[\mathcal{L}_{ri}]\ast({_ij}_{0})_![\cL'_r],$$
where
$[\cL_r]={_i\mathrm{res}}_{\mathbf{W}_r}^{\mathbf{V}}[{_ij}^\ast_{\mathbf{V},r}\mathcal{L}]$,
$[\cL'_r]={_i\mathrm{res}}_{\mathbf{W}_r}^{\mathbf{V}}[{_ij}^\ast_{\mathbf{V},r}\mathcal{L}']$ and $\mathbf{W}_r$ is an $I$-graded $\mathbb{K}$-vector space with $\dim W_i=\dim V_i-r$ and $\dim W_j=\dim V_j$ for all $j\neq i$.

Define
$$\tilde{T}_i([\cL]^-)=\sum_{r\geq0}v^{\langle{\nu,ri}\rangle}\mathbf{k}_{-rh_i}[\cL_{ti}]^+({^ij}_{0})_![\tilde{\omega}_i(\cL_r)]^-,$$
$$\tilde{T}_i([\cL']^+)=\sum_{r\geq0}v^{\langle{\nu,ri}\rangle}\mathbf{k}_{rh_i}[\cL_{ti}]^-({^ij}_{0})_![\tilde{\omega}_i(\cL'_r)]^+$$
and
$$\tilde{T}_i(\mathbf{k}_{\mu})=\mathbf{k}_{s_i(\mu)}.$$
Hence, we get a map $$\tilde{T}_i:DK(\mathcal{Q})(Q)\rightarrow DK(\mathcal{Q})(Q').$$

\begin{proposition}\label{bij}
The map $\tilde{T}_i$ is a bijection.
\end{proposition}

For the proof of this proposition, we should construct the inverse of $\tilde{T}_i$.
The algebra $DK(\mathcal{Q}))(Q')$ also has the following  $\mathcal{A}$-basis
$$\{[\cL]^-\mathbf{k}_{\mu} [\cL']^+\},$$
where
$\cL\in\mathcal{P}_{\mathbf{V}}$ for some $I$-graded $\mathbb{K}$-vector space $\mathbf{V}$ with $\underline{\dim}\bfV=\nu\in\bbN I$,
$\cL'\in\mathcal{P}_{\mathbf{V}'}$ for some $I$-graded $\mathbb{K}$-vector space $\mathbf{V}'$ with $\underline{\dim}\bfV'=\nu'\in\bbN I$
and $\mu\in P^{\vee}$.

By Theorem \ref{decom2}, we have
$$[\cL]=\sum_{r\geq0}({^ij}_{0})_![\cL_r]\ast[\mathcal{L}_{ri}]$$
and
$$[\cL']=\sum_{r\geq0}({^ij}_{0})_![\cL'_r]\ast[\mathcal{L}_{ri}],$$
where
$[\cL_r]={^i\mathrm{res}}_{\mathbf{W}_r}^{\mathbf{V}}[{^ij}^\ast_{\mathbf{V},r}\mathcal{L}]$,
$[\cL'_r]={^i\mathrm{res}}_{\mathbf{W}_r}^{\mathbf{V}}[{^ij}^\ast_{\mathbf{V},r}\mathcal{L}']$.

Define
$$\tilde{T}'_i([\cL]^-)=\sum_{r\geq0}v^{\langle{ri,\nu}\rangle}({_ij}_{0})_![\tilde{\omega}^{-1}_i(\cL_r)]^-[\cL_{ti}]^+\mathbf{k}_{rh_i},$$
$$\tilde{T}'_i([\cL']^+)=\sum_{r\geq0}v^{\langle{ri,\nu}\rangle}({_ij}_{0})_![\tilde{\omega}^{-1}_i(\cL'_r)]^+[\cL_{ti}]^-\mathbf{k}_{-rh_i}$$
and
$$\tilde{T}'_i(\mathbf{k}_{\mu})=\mathbf{k}_{s_i(\mu)}.$$
Hence, we get a map $$\tilde{T}'_i:DK(\mathcal{Q})(Q')\rightarrow DK(\mathcal{Q})(Q).$$
By the definitions of $\tilde{T}_i$ and $\tilde{T}'_i$, the map $\tilde{T}'_i$ is the inverse of $\tilde{T}_i$. Hence, we have proved Proposition \ref{bij}.

The following theorem is the main result in this paper.

\begin{theorem}\label{main}
The map $\tilde{T}_i:DK(\mathcal{Q})(Q)\rightarrow DK(\mathcal{Q})(Q')$ is an isomorphism of Hopf algebras satisfying that
the following diagram is commutative
$$
\xymatrix{
{DK(\mathcal{Q})(Q)}\ar[d]^{\cong}\ar[r]^{\tilde{T}_i}&{DK(\mathcal{Q})(Q')}\ar[d]^{\cong}\\
{\mathbf{U}_{\mathcal{A}}}\ar[r]^{T_i}&{\mathbf{U}_{\mathcal{A}}}.}
$$
\end{theorem}
\begin{proof}
Consider the $\mathcal{A}$-basis ${_i\mathbf{B}}_{0}$ of $K({_i\mathcal{Q}}_{0})$.
By Theorem \ref{decom}, the following set is an $\mathcal{A}$-basis of $DK(\mathcal{Q})(Q)$
$$\{[\mathcal{L}_{ri}]^-({_ij}_{0})_!(b)^-\mathbf{k}_{\mu}[\mathcal{L}_{r'i}]^+({_ij}_{0})_!(b')^+\,\,|
\,\,r,r'\in\mathbb{Z}_{\geq0},b,b'\in{_i\mathbf{B}}_{0},\mu\in P^{\vee}\}.$$
Similarly, by Theorem \ref{theorem:decom}, the following set is an $\mathcal{A}$-basis of $\mathbf{U}_{\mathcal{A}}$
$$\{{F^{(r)}_i}{{_i\lambda}_{0,\mathcal{A}}}(b)^-K_{\mu}{E^{(r')}_i}{{_i\lambda}_{0,\mathcal{A}}}(b')^+\,\,|
\,\,r,r'\in\mathbb{Z}_{\geq0},b,b'\in{_i\mathbf{B}}_{0},\mu\in P^{\vee}\}.$$
Note that these two basis are identified under the isomorphism $\lambda_{\mathcal{A}}:DK(\mathcal{Q})(Q)\cong\mathbf{U}_{\mathcal{A}}$.
By Theorem \ref{theorem:Lusztig's_algebra} and the definition of $\tilde{T}_i$, we have the desired commutative diagram.
Since $T_i:\mathbf{U}_{\mathcal{A}}\rightarrow \mathbf{U}_{\mathcal{A}}$ is an isomorphism of Hopf algebras, so is the map $\tilde{T}_i:DK(\mathcal{Q})(Q)\rightarrow DK(\mathcal{Q})(Q')$.
\end{proof}

\section{Braid group relations}

In this section, we shall consider the braid group relations of Lusztig's symmetries.

\subsection{}
First, we shall recall the Fourier-Deligne transform (\cite{Kiehl_Weissauer_Weil_conjectures_perverse_sheaves_and_l'adic_Fourier_transform}\cite{Lusztig_Introduction_to_quantum_groups}).
Let $Q=(I,H,s,t)$ be a quiver.
Let $E$ be a subset of $H$ and denote by $Q'=\sigma_EQ$ the quiver obtained from $Q$ by reversing all
the arrows in $E$. Given $\nu\in\bbN I$, let $\bfV$ be an $I$-graded $\mathbb{K}$-vector space with dimension vector $\nu$.
The Fourier-Deligne transform is denoted by (\cite{Lusztig_Introduction_to_quantum_groups})
$$
\Theta_{Q,Q'}:\cD^b_{G_{\bfV}}(E_{\bfV,Q})\rightarrow\cD^b_{G_{\bfV}}(E_{\bfV,Q'}).
$$

\begin{lemma}[\cite{Lusztig_Introduction_to_quantum_groups}]\label{fD}
The Fourier-Deligne transform $\Theta_{Q,Q'}:\cD^b_{G_{\bfV}}(E_{\bfV,Q})\rightarrow\cD^b_{G_{\bfV}}(E_{\bfV,Q'})$ is an equivalence of triangulated categories and
$\Theta_{Q,Q'}(\mathcal{Q}_{\mathbf{V},Q})=\mathcal{Q}_{\mathbf{V},Q'}$.
\end{lemma}

%


By Lemma \ref{fD}, we have the following proposition.

\begin{proposition}
The Fourier-Deligne transform $\Theta_{Q,Q'}:\mathcal{Q}_{\mathbf{V},Q}\rightarrow\mathcal{Q}_{\mathbf{V},Q'}$ for various dimension vectors induce an isomorphism of Hopf algebras $$\Theta_{Q,Q'}:DK(\mathcal{Q})(Q)\rightarrow DK(\mathcal{Q})(Q')$$ sending $[\cL]^-\mathbf{k}_{\mu} [\cL']^+$ to $[\Theta_{Q,Q'}(\cL)]^-\mathbf{k}_{\mu} [\Theta_{Q,Q'}(\cL')]^+.$
\end{proposition}

\subsection{}

For any quiver $Q=(I,H,s,t)$, choose a subset $E$ of $H$ such that $i$ is a sink of $\sigma_EQ$. Consider the quiver $\sigma_i\sigma_EQ=(I,H',s,t)$ with $i$ as a source. Choose a subset $E'$ of $H'$ such that $\sigma_{E'}\sigma_i\sigma_EQ=Q$.

Define $$\hat{T}_i:DK(\mathcal{Q})(Q)\rightarrow DK(\mathcal{Q})(Q)$$
as the composition
of
$$
\xymatrix@=35pt{
K(\mathcal{Q})(Q)\ar[r]^-{\Theta_{Q,\sigma_EQ}}&K(\mathcal{Q})(\sigma_EQ)\ar[r]^-{\tilde{T}_i}&DK(\mathcal{Q})(\sigma_i\sigma_EQ)\ar[r]^-{\Theta_{\sigma_i\sigma_EQ,Q}}&K(\mathcal{Q})(Q)}
.
$$


By using Fourier-Deligne transform, Theorem \ref{main} can be rewrote as following.

\begin{theorem}
The map $\hat{T}_i:DK(\mathcal{Q})(Q)\rightarrow DK(\mathcal{Q})(Q)$ is an isomorphism of Hopf algebras satisfying that
the following diagram is commutative
$$
\xymatrix{
{DK(\mathcal{Q})(Q)}\ar[d]^{\cong}\ar[r]^{\hat{T}_i}&{DK(\mathcal{Q})(Q)}\ar[d]^{\cong}\\
{\mathbf{U}_{\mathcal{A}}}\ar[r]^{T_i}&{\mathbf{U}_{\mathcal{A}}}.}
$$
\end{theorem}

%

Since ${T}_i:\mathbf{U}\rightarrow\mathbf{U}$ satisfies the braid group relations, the isomorphism $\hat{T}_i:DK(\mathcal{Q})(Q)\rightarrow DK(\mathcal{Q})(Q)$ also satisfies the braid group relations, that is, we have the following commutative diagrams
$$
\xymatrix@=15pt{
&{DK(\mathcal{Q})(Q)}\ar[r]^{\hat{T}_j}&{DK(\mathcal{Q})(Q)}\ar[rd]^{\hat{T}_i}&\\
{DK(\mathcal{Q})(Q)}\ar[ru]^{\hat{T}_i}\ar[rd]^{\hat{T}_j}&&&{DK(\mathcal{Q})(Q)}\\
&{DK(\mathcal{Q})(Q)}\ar[r]^{\hat{T}_i}&{DK(\mathcal{Q})(Q)}\ar[ru]^{\hat{T}_j}&}
$$
for any $i\neq j\in I$ such that $a_{ij}=-1$, and
$$
\xymatrix@=15pt{
&{DK(\mathcal{Q})(Q)}\ar[rd]^{\hat{T}_j}&\\
{DK(\mathcal{Q})(Q)}\ar[ru]^{\hat{T}_i}\ar[rd]^{\hat{T}_j}&&{DK(\mathcal{Q})(Q)}\\
&{DK(\mathcal{Q})(Q)}\ar[ru]^{\hat{T}_i}&}
$$
for any $i\neq j\in I$ such that $a_{ij}=0$.


\bibliography{mybibfile}

\end{document}